\documentclass[12pt]{amsart}
\usepackage[foot]{amsaddr}   
\usepackage{graphicx, amssymb,amsmath, latexsym,cite,fancyhdr}
\usepackage[dvipsnames]{xcolor}
\usepackage{amsthm,arydshln}
\usepackage{times}      
\usepackage[T1]{fontenc}      
\usepackage{url}
\usepackage{mathabx,xfrac}    

\usepackage{graphicx} 
\usepackage{booktabs} 
\usepackage{xparse}   
\NewDocumentCommand{\rot}{O{20} O{1em} m}{\makebox[#2][l]{\rotatebox{#1}{#3}}}%

\usepackage[ruled,vlined]{algorithm2e}

\newtheoremstyle{mythm} 
{6pt}
{6pt}
{\it}
{}
{\bf}
{.}
{.5em}
{}

\newtheoremstyle{mydef}
{6pt}
{6pt}
{}
{}
{\bf}
{.}
{.5em}
{}

\newtheoremstyle{myrem}
{6pt}
{6pt}
{}
{}
{\bf}
{.}
{.5em}
{}

\newtheoremstyle{myex}
{6pt}
{3ex}
{}
{}
{\bf}
{.}
{.5em}
{}

\theoremstyle{mythm}
\newtheorem{theorem}{Theorem}[section]
\newtheorem{lemma}[theorem]{Lemma}

\theoremstyle{mydef}
\newtheorem{definition}[theorem]{Definition}

\theoremstyle{myrem}
\newtheorem{remark}[theorem]{Remark}
\numberwithin{equation}{section}

\theoremstyle{myex}
\newtheorem{example}[theorem]{Example}

\newcommand{\Aut}{{\rm Aut}}

\newcommand{\opp}{{\nu}}
\newcommand{\clust}{\mathcal{C}}
\newcommand{\Hom}{{\rm Hom}}

\newcounter{ithmcount}

\newenvironment{items}{
\begin{list}{$\alph{item})$}
{\labelwidth30pt \leftmargin30pt \topsep3pt \itemsep2pt \parsep0pt}}
{\end{list}}

\newenvironment{iprfwithoutbox}{\begin{list}{{\rm
	\alph{ithmcount})}}{\usecounter{ithmcount}\labelwidth-5pt
      \leftmargin0pt \topsep3pt \itemsep1pt \parsep2pt}}{\end{list}}

\allowdisplaybreaks

\renewcommand{\leq}{\leqslant}
\renewcommand{\geq}{\geqslant}

\begin{document}


\title{Generation of finite groups with cyclic Sylow subgroups}
 
\author[H.~Dietrich]{Heiko Dietrich}
\email{heiko.dietrich@monash.edu}
\author[D.\ Low]{Darren Low}
\email{darrenlow@protonmail.com}
\address{School of Mathematics, Monash University, VIC 3800, Australia}
\thanks{The first author was supported by an Australian Research Council grant,
identifier DP190100317. This work forms part of the Honours thesis \cite{dlh} of the second author.}

\begin{abstract}
  Slattery (2007) described computational methods to enumerate, construct, and identify finite groups of squarefree order. We generalise Slattery's result to the class of finite groups that have cyclic  Sylow subgroups and provide an implementation for the computer algebra system~GAP.   
\end{abstract} 

\maketitle


\section{Introduction}\label{s:1}
\noindent A finite group is squarefree if its order is not divisible by any prime-square. H\"older \cite{hoelder} classified and enumerated these groups at the end of the 19th century. More than 100 years later, Slattery \cite{sf} devised powerful algorithms for constructing and identifying groups of squarefree order. By describing a \emph{canonical ordering} for the (isomorphism types of) groups of squarefree order~$n$,  Slattery's methods allow the direct computation of the $i$-th group of order $n$, and, for a given squarefree group of order $n$, the determination of its position in that ordered list, both tasks without the need to compute the full list of groups. Slattery's algorithms are implemented for the computer algebra system Magma \cite{magma}; the SmallGroups Library (based on \cite{sgd}) of the computer algebra system GAP \cite{gap} contains alternative implementations for squarefree groups.  Clearly, a group of squarefree order has cyclic Sylow subgroups. The aim of this note is to generalise  \cite{sf} to the class of C-groups, that is, to finite groups having cyclic Sylow subgroups. 

\subsection{Main results} 
For a positive integer $n$ denote by $C(n)$ the number of (isomorphism types) of C-groups of order $n$; a formula for  $C(n)$ has been determined by Murty \& Murty \cite{murty}, see Section~\ref{secCGroups}. Our first main result is a complete generalisation of the results in \cite{sf} to C-groups; specifically, for a positive integer $n$ we   describe algorithms for the following:  
\begin{items}
\item[(A)] Construct a sorted list $\mathcal{C}_n$ of all (isomorphism types of) C-groups of order $n$.
\item[(B)] For $i\in\{1,\ldots,C(n)\}$ construct the $i$-th group in $\mathcal{C}_n$ without creating the whole list.
\item[(C)] For a C-group of order $n$, determine its position $i$ in $\mathcal{C}_n$ without creating the whole list.
\end{items}
This functionality furnishes every  C-group with a unique ID $(n,i)$, meaning that it is isomorphic to the $i$-th group in $\mathcal{C}_n$, and this ID is an isomorphism invariant. This has two advantages. Firstly, it allows us to efficiently reduce a given list of C-groups up to isomorphism, by simply computing and comparing IDs. Secondly, it provides a dynamic database of C-groups that allows a direct construction of groups without the necessity to store all groups. Our second main result is an implementation for GAP. Even though our implementation started out  as a proof of concept, it is remarkably efficient: restricted to squarefree groups, our implementations seem on average faster than what is provided by GAP's SmallGroups Library: we even found squarefree group orders for which our code runs more than 800 times faster than the latter, see Table \ref{tabSQ}. In GAP, our groups are constructed as PC-groups (polycyclic groups), but our ID function also accepts other input formats such as permutation or matrix groups.

We proceed as follows: In Section~\ref{secCGroups} we discuss relevant results for C-groups. In Section~\ref{secConst} we describe the theory for  (A), (B), and (C). In Section~\ref{secGAP} we illustrate the capabilities of our GAP implementation.


\section{C-groups}\label{secCGroups}
\noindent We present some known and preliminary  results for C-groups. To the best of our knowledge, the term \emph{C-group} was coined by Murty \& Murty~\cite{murty}; some authors call them Z-groups, referencing  Zassenhaus' classification.

\subsection{Structure results} 
All groups are finite and, throughout, $C_k$ denotes a cyclic group of size~$k$. Results of H\"older, Burnside, and Zassenhaus \cite[(10.1.10)]{rob} show that every C-group $G$ of order $n$ is metacyclic with odd-order derived subgroup $G'\cong C_b$ and cyclic quotient $G/G'$ of order $a=n/b$; specifically,  $G$  is isomorphic to
\begin{eqnarray}\label{eqGedr}G_{a,b,r} = \langle x,y \mid x^a, y^b, y^x=y^r\rangle
\end{eqnarray}
for some $0\leq r<b$ with $r^a\equiv 1\bmod b$ and $\gcd(a(r-1),b)=1$. Conversely, every such $G_{a,b,r}$ is a C-group of order $ab$ with derived subgroup isomorphic to $C_b$; note that $G_{a,b,r}$ is abelian if and only if $b=1$ and $r=0$. Murty \& Murty have shown in \cite[Theorem 1.1]{murty} that the number (of isomorphism types) of C-groups of order $n$ is 
\begin{eqnarray}\label{eqMurty}C(n)=\sum\nolimits_{n=de \atop \gcd(d,e)=1} \prod\nolimits_{p^\alpha|| d\atop p\text{ prime}} \left(
\sum\nolimits_{j=1}^{\alpha} \frac{(p^{\opp(p^j,e)} - p^{\opp(p^{j - 1}, e})}
                               {p^{j-1}(p-1)}\right)
\end{eqnarray} where \[p^{\opp(p^j, e)} = \prod\nolimits_{q\mid e \atop q\text{ prime}} \gcd (p^j, q - 1)\] and $p^\alpha|| d$ denotes the largest $p$-power dividing $d$,  which we also write as  $d_p=p^\alpha$.

\begin{lemma}\label{lemIsom}Two groups $G_{a,b,r}$ and $G_{a',b',r'}$ are isomorphic if and only if $a=a'$, $b=b'$, and $r^\alpha\equiv r'\bmod b$ for some $\alpha$ coprime to $a$.
\end{lemma}
\begin{proof}
If $G\cong G_{a,b,r}$, then $|G'|=b$ and $|G/G'|=a$, so $a$ and $b$ are uniquely determined by the isomorphism type of $G$. Now  \cite[Lemma 3.6]{murty} proves the claim.
\end{proof} 

\begin{remark}\label{remBF}The description in \cite[(10.1.10)]{rob} and Lemma \ref{lemIsom} allow, in principle,  a brute-force construction of all C-groups of order $n$ up to isomorphism and, for a given C-group $G$, a determination of parameters $a,b,r$ such that $G\cong G_{a,b,r}$.  However, from the perspective of complexity theory, these are \emph{difficult} tasks, and this is reflected in the practical performance of the brute-force approach. Iliopulous \cite{Iliopoulos} demonstrated that even for abelian black-box groups\footnote{See \cite[Section~2]{Seress} for an introduction to the black-box model of computation.} most problems are at least as hard as the discrete logarithm problem and 
integer factorisation (neither of those problems seems to have  a polynomial time 
deterministic or randomised solution); this includes the problem of finding a generator of a cyclic group given by some generating set, see \cite[Proposition~2.5]{Iliopoulos}, which is relevant when determining the parameter $b$ for $G\cong G_{a,b,r}$. Even if a group is given by a polycyclic presentation, which is the standard way in computational group theory to describe finite solvable groups, proving that its group multiplication  has  a favourable complexity is already difficult because of the challenges involving \emph{collection}, see \cite{NNcomp}.   
\end{remark}

For the reasons given in Remark \ref{remBF}, we ignore complexity issues here, and we refer to \cite{most} for a recent discussion of the complexity of (black-box) isomorphisms tests. Instead, as in \cite{sf}, we assume that we can factorise integers and perform basic  computations with groups, such as constructing group orders, Sylow subgroups, quotient groups, cyclic generators, etc; we refer to the Handbook of Computational Group Theory \cite{handbook} for details on how to compute with groups. Analogous to \cite{sf}, our approach to construct and identify C-groups with $|G'|=d$ and $|G/G'|=e$ is to iterate  over the prime-power divisors of $d$ and $e$; in practice, this improves significantly over the brute-force approach (Remark~\ref{remBF}), cf.\ the computational evidence in Section~\ref{secGAP}.

\subsection{Permissible sets and clusters}

The next two definitions and following lemma generalise \cite[Definitions 4 and 7, Lemma 5]{sf} and is motivated by \cite[p.\ 302]{murty}.

\begin{definition}
  Let $G$ be a C-group and let  $q>p$ be prime divisors  of $|G|$. We say \emph{$p$ acts on $q$ (in $G$ with exponent $e$)} if a Sylow $p$-subgroup of $G$ acts via conjugation on a Sylow $q$-subgroup of $G$, and this action describes an automorphism of order $p^e\ne 1$.
\end{definition}
 
\begin{remark}\label{remPowers}
 Let $q>2$ be a prime and $m\geq 1$ an integer, so that the automorphism group of  $Q=\langle y\rangle\cong C_{q^m}$ is cyclic of order $q^{m-1}(q-1)$, generated by an automorphism $y\mapsto y^r$, where $r$ is a primitive root modulo $q^m$. For every divisor $a$ of the automorphism group order, we now describe a canonical ordering of the automorphisms of $Q$ of order $a$. For this we  choose $r\in\{2,\ldots,q^m-1\}$ as small as possible and for $k\geq 1$ define $\alpha_k\colon Q\to Q$, $y\mapsto y^{t_k}$, where $t_k=r^{kq^{m-1}(q-1)/a}\bmod q^m$. The automorphisms of $Q$ of order $a$ can now be ordered  as $\{\alpha_k : k\in\{1,\ldots,a-1\}\text{ coprime to $a$}\}$.
\end{remark}

If a group $A$ acts on a group $B$ via $\varphi\colon A\to\Aut(B)$, then  $A\ltimes_\varphi B$ denotes the corresponding split extension with underlying set $A\times B$ and multiplication $(a,b)(u,v)=(au,b^{\varphi(u)}v)$. We also write $A\ltimes B$ if the action is not specified or understood from the context.

\begin{lemma}\label{lemDivCond}
Every C-group $G$ decomposes as  $G=(A\ltimes D)\times C$ where $A\times C\cong G/G'$ and $D\cong G'$; the groups $A,C,D$ are cyclic of coprime orders, and $p$ divides $|A|$ if and only if $p$ acts on some $q$ in $G$ with exponent $e$; in this case,  $p^e$ divides $q-1$, and $q$ divides $|D|$.
\end{lemma}
\begin{proof}
  The classification \eqref{eqGedr} shows that $G=U\ltimes D$ where $U\cong G/G'$ and $D\cong G'$ are cyclic of coprime orders. Write $U=A\times C$ where $A$ is the direct product of those Sylow $p$-subgroups of $U$ that have $p$ acting on some $q$ in $G$, and $C$ is the direct product of the remaining Sylow subgroups of $U$. By construction,  $C\leq Z(G)$ and $A\ltimes D$ is a complement to $C$ in $G$.  Let $P$ be a Sylow $p$-subgroup acting on a Sylow $q$-subgroup $Q$ of $G$ with exponent $e$. If $Q$ has order $q^m$, then $\Aut(Q)$ is abelian of order $q^{m-1}(q-1)$; since $P$ acts with exponent $e$,  we have $p^e\mid q-1$.
  If $p$ divides $|D|$, then $P\leq D$, because the orders of $A$, $C$, and $D$ are coprime and  $D\unlhd G$. In particular, $P$ is the unique Sylow $p$-subgroup of $D\unlhd G$, and hence $P$ and $Q$ normalise each other. This implies that   $x^{-1}y^{-1}xy\in P\cap Q=\{1\}$ for all $x\in P$ and $y\in Q$, so $P$ centralises $Q$, a contradiction. Hence $p$ divides $|U|$. Since $C\leq Z(G)$ and $P$ acts nontrivially on $Q$, we have $p\mid |A|$ and $q\mid |D|$.\end{proof}

\begin{definition}
  Let $G=(A\ltimes D)\times C$ be  a C-group as in Lemma \ref{lemDivCond}. Its \emph{acting divisor} is $|A|$; its \emph{cluster} is the set $\clust(G)$ of all triples $(p,q,e)$ such that $p$ acts on $q$ in $G$ with exponent~$e$. The \emph{acting group} of $G$ is the cyclic subgroup of $\Aut(D)$ generated by the $A$-action on $D$.  A \emph{permissible set} for a group order $n>1$ is a set $\mathcal{P}$ of triples $(p,q,e)$ such that $q>p$ are primes dividing $n$, $p^e\ne 1$ divides $q-1$,  and if  $(p,q,e)\in\mathcal{P}$, then $(q,\ast,\ast),(\ast,p,\ast),(p,q,c)\notin \mathcal{P}$ for all $c\ne e$. 
\end{definition}

We conclude with a few remarks

\begin{remark}\label{remCluster}
  \begin{iprfwithoutbox}
  \item If $G$ is a C-group, then $\clust(G)$ is a permissible set for $|G|$ by Lemma~\ref{lemDivCond}. The acting divisor of $G$ is the product of all $|G|_p$ where $p$ runs over $\{p : (p,\ast,\ast)\in\clust(G)\}$; every such $p$ divides the size of the acting group of $G$.
Conversely, if $\mathcal{P}$ is a permissible set for a group order $n>1$, then there exists a C-group $G$ with $\clust(G)=\mathcal{P}$: let $n=p_1^{a_1}\ldots p_k^{a_k}q_1^{b_1}\ldots q_\ell^{b_\ell}r_1^{c_1}\ldots r_m^{c_m}$ be the prime power factorisation of $n$ and assume, without loss of generality, that $(p,\ast,\ast),(\ast,q,\ast)\in\mathcal{P}$ if and only if $p\in\{p_1,\ldots,p_k\}$ and $q\in\{q_1,\ldots,q_\ell\}$. A group of order $n$ with cluster $\mathcal{P}$ is\[G=(C_{p_1^{a_1}}\times \ldots\times C_{{p_k^{a_k}}})\ltimes (C_{{q_1}^{b_{1}}}\times\ldots\times C_{{q_\ell}^{b_\ell}})\times C_{r_1^{c_1}\ldots r_m^{c_m}}\]
where $C_{{p_i}^{a_i}}$ acts on $C_{{q_j}^{b_j}}$ via an automorphism of order $p_i^{e_{i,j}}$ if and only if $(p_i,q_j,e_{i,j})\in\mathcal{P}$.
\item If two $C$-groups $G$ and $\tilde G$ have equal orders and clusters, then  we can identify the decomposition sets in Lemma \ref{lemDivCond}, so the acting groups of $G$ and $\tilde G$ are both subgroups of  $\Aut(D)=\Aut(\tilde D)$. Using this identification, it follows with Lemma \ref{lemIsom} and \cite[Lemma 6.14b)]{most} that two C-groups $G$ and $\tilde G$ are isomorphic if and only if they have the same order, clusters, and acting groups. 
\item If $G$ is a C-group with largest prime divisor $q$, then the Sylow $q$-subgroup $Q$ of $G$ is normal, and $G=H\ltimes Q$ for some C-group $H<G$ by the Schur-Zassenhaus theorem \cite[(9.1.2)]{rob}. By induction, $G$ is an iterated split extension of Sylow subgroups of increasing primes, that is, $G$ has a split Sylow tower.
  \end{iprfwithoutbox}
\end{remark}

\subsection{Isomorphic split extensions}

Motivated by Remark \ref{remCluster}c), we require the following. 
\begin{lemma}\label{lemIsoExt}Let $H$ and $Q$ be groups of coprime order. Let $\sigma,\omega\colon H\to \Aut(Q)$ be two group actions and suppose that $Q$ and $\Aut(Q)$ are abelian. Then $H\ltimes_\sigma Q$ and $H\ltimes_\omega Q$ are isomorphic if and only if there exists $\alpha\in\Aut(H)$ with  $\omega=\sigma^{\alpha}$ where $\sigma^{\alpha}(h)=\sigma(h^{\alpha^{-1}})$ for $h\in H$.
\end{lemma}
\begin{proof} Write $E_{1} = H\ltimes_{\sigma}Q$ and $E_{2}= H\ltimes_{\omega} Q$ and recall that $E_1$ and $E_2$ share the same underlying set $H\times Q$.   First, suppose  $\omega = \sigma^\alpha$, that is, $\omega(h^{\alpha}) = \sigma(h)$ for all $h\in H$. A straightforward calculation shows that  $E_{1}\rightarrow E_{2}$, $(h,n)\mapsto (h^\alpha, n)$, is an isomorphism. Conversely, suppose $\beta \colon E_{1}\rightarrow E_{2}$ is an isomorphism. Since $Q$ is a characteristic Sylow $q$-subgroup in both $E_{1}$ and $E_{2}$, there is a map $\tau\colon H\to Q$ such that $\beta$ has the form $(h,n)^\beta=(h^\alpha, h^\tau n^\beta)$, where $\alpha=\beta|_{E_1/Q}$ is the induced automorphism of $H$. Since $\beta$ is an isomorphism, a calculation shows that $(gh)^\tau=g^{\tau\omega(h^\alpha)}h^\tau$ and  $n^{\sigma(h)\beta}=n^{\beta\omega(h^\alpha)}$ for all $n\in Q$ and $g,h\in H$. Since $\Aut(Q)$ is abelian,  $\sigma(h)=\omega(h^\alpha)$ for all $h\in H$, as claimed. 
\end{proof} 

\begin{remark}\label{remInduced}
  We apply Lemma \ref{lemIsoExt} to the set-up of Remark \ref{remCluster}c), that is, $H$ is a C-group and $Q$ is a cyclic group. In order to describe all split extensions $H\ltimes Q$ up to isomorphism, one has to compute the $\Aut(H)$-orbits on $\Hom(H,\Aut(Q))$. Since $\Aut(Q)$ is abelian, every such homomorphism  has $H'$ in its kernel, and we can identify \[\Hom(H,\Aut(Q))=\Hom(H/H',\Aut(Q))=\bigoplus\nolimits_{j=1}^{\ell}\Hom(S_j,\Aut(Q)),\] where $S_1,\ldots,S_\ell$ are the Sylow subgroups of $H/H'$. Note that  each  $\Hom(S_j,\Aut(Q))$ is fixed by $\Aut(H)$. It follows that if $\sigma,\omega\in\Hom(H,\Aut(Q))$ are in the same $\Aut(H)$-orbit, then  for each $j$ the induced actions $\sigma_{j},\omega_{j}\in\Hom(S_j,\Aut(Q))$ are in the same $\Aut(H)$-orbit.
\end{remark}

\section{Construction of C-groups}\label{secConst}
\noindent We describe how to construct, up to isomorphism, all C-groups of order $n>1$; recall from \eqref{eqMurty} that there are $C(n)$  such groups. This construction will be \emph{canonical}, that is, every C-group $G$ of order $n$ has a unique, isomorphism-invariant ID $(n,i)$, meaning that  $G$ is isomorphic to the $i$-th group in that list. In Section \ref{secConstID} we describe the direct construction of the group with ID $(n,i)$;  in Section \ref{secId} we describe the converse and determine the ID of a given C-group. To facilitate these constructions, we first comment on a canonical ordering of permissible sets.

\subsection{Murty's counting formulas}
Let $n=p_1^{a_1}\ldots p_k^{a_k}$ be the prime-power factorisation of $n$ with $p_1<\ldots<p_k$. Let $\mathcal{D}(n)$ be the set of possible acting divisors of C-groups of order $n$; these are the $2^{k-1}$ numbers $p_{i_1}^{a_{i_1}}\ldots p_{i_m}^{a_{i_m}}$ where $\{i_1,\ldots,i_m\}$ runs over all subsets of $\{1,\ldots,k-1\}$. We sort them as $d_1<\ldots< d_{2^{k-1}}$. For $d\in\mathcal{D}(n)$ denote by $C_d(n)$ the number of (isomorphism types) of C-groups of order $n$ with acting divisor $d$; by \eqref{eqMurty} and \cite[pp.\ 303--304]{murty}, we have
\[C_d(n)= \prod\nolimits_{p^\alpha|| d \atop p\text{ prime}} \left( 
\sum\nolimits_{j=1}^{\alpha} \frac{(p^{\opp(p^j,n/d)} - p^{\opp(p^{j - 1}, n/d})}
                               {p^{j-1}(p-1)}\right).
                               \]
                               Moreover, let $C_{d,m}(n)$ be the number of (isomorphism types) of C-groups of order $n$ with acting divisor $d$ and acting group order $m$; then \cite[p.\ 303]{murty} shows that
\[C_{d,m}(n)= \prod\nolimits_{p^j|| m \atop p\text{ prime}}\left( \frac{(p^{\opp(p^j,n/d)} - p^{\opp(p^{j - 1}, n/d})}
                               {p^{j-1}(p-1)}\right).
                               \]
Recall that if $d$ is the acting divisor and $m$ is the size of the acting group, then $m\mid d$ and, if $p\mid d$ is a prime, then $p\mid m$. This restricts the orders of possible  acting groups.
 
\subsection{Ordering of permissible sets}\label{secSortCluster} A permissible set $\mathcal{P}$ is a cluster for a C-group of order $n$ with acting divisor $d$ and acting group order $m$ if and only if for every $p^j|| m$ there exists $(p,*,j)\in \mathcal{P}$, but $(p,*,i)\notin\mathcal{P}$ when $i>j$. For every such $(p,*,j)$ we note which prime divisors $q$ of $n/d$ may be acted upon by $p$ with maximal exponent $e$, that is, $p^e=\gcd(p^j,q-1)\neq1$. Having done this, we are able to sort such permissible sets lexicographically, with four tiers of ordering:  
\begin{enumerate} 
\item compare the acting primes;
\item for a fixed acting prime $p$, compare the first prime $q_{\max}$ it acts on with maximal exponent;
\item for a fixed acting prime $p$ and $q_{\max}$, compare the other primes $p$ acts on;
\item for a fixed acting prime $p$ acting on $q\ne q_{\max}$, compare the exponents of that action.
\end{enumerate}
\begin{example}We sort the permissible sets for $n=2^2.3.5.13$ with $d=m=4$. Only $2$ can act and, as $m=4$, it has to act at least once with exponent 2. Clusters containing  $(2,5,2)$ (case ``$q_{\text{max}}=5$'') come before those containing $(2, 13, 2)$ but not $(2, 5, 2)$ (``$q_{\text{max}}=13$''). We obtain the following sorted list:  $\{(  2,  5,  2 )\}$,\;\; $\{(   2,   5,   2 ), (   2,  13,   1 )\}$,\;\; $\{(2,   5,   2), (  2,  13,   2 )\}$,\;\; $\{(  2,  3,  1 ),(  2,  5,  2 )\}$,\;\; $\{(  2,   3,   1 ), (   2,   5,   2 ),(   2,  13,   1 )\}$,\;\; $\{(2,   3,   1 ),(   2,   5,   2 ),(   2,  13,   2 )\}$,\;\; $\{( 2,  13,   2 )\}$,\;\; $\{(   2,   3,   1 ),(   2,  13,   2 )\}$,\;\; $\{(   2,   5,   1 ),(   2,  13,   2 )\}$,\;\; $\{( 2,   3,   1 ),(   2,   5,   1),(   2,  13,   2 )\}$. The triples within a fixed permissible set are simply arranged first by the primes which act, then by the primes which are acted on.
\end{example}
 
\subsection{Generating groups with a given cluster}\label{secGenP} Now that we have imposed a canonical ordering on the set of all permissible sets of a given group order, it remains to canonically sort the (isomorphism types of) groups that correspond to such data. We describe this process in the following. Throughout, let $\mathcal{P}$ be a permissible set for a group order $n$, with acting divisor $d$ and acting group order $m$. We factorise  $n=p_1^{a_1}\ldots p_k^{a_k}$ with $p_1<\ldots<p_k$.
 
We use the extension procedure described in Remark \ref{remCluster}c) to construct all C-groups $G$ of order $n$ with $\clust(G)=\mathcal{P}$. Recall that every such $G$ has a split Sylow tower $T_1\ltimes \ldots \ltimes T_k$, where each $T_i$ is a Sylow $p_i$-subgroup. By induction, we can assume that we have constructed all (isomorphism types) of C-groups $H$ of order $p_1^{a_1}\ldots p_{v-1}^{a_{v-1}}$ with cluster $\clust(H)= \{(p,q,c)\in \mathcal{P} : p,q < p_{v}\}$; the induction base case is   $H=T_1=C_{p_1^{a_1}}$. We now have to consider all split extensions $H\ltimes Q$ with $Q\cong C_{p_v^{a_v}}$ whose cluster is induced by $\mathcal{P}$. This requires a case distinction involving \[\mathcal{A}=\{p :  (p,p_v,*)\in\mathcal{P}\},\]the set of primes $p\in\{p_1,\ldots,p_{v-1}\}$ that act on $p_v$ in $\mathcal{P}$. 

If $\mathcal{A}=\emptyset$, then the only split extension we have to consider is the direct product $H\times Q$.\linebreak Now suppose $\mathcal{A}=\{p_{i_1},\ldots,p_{i_a}\}\ne\emptyset$, that is, $\mathcal{P}$ contains triples $(p_{i_j},p_v,e_j)$ for each $j$. For each $j$, let $P_j=\langle x_j\rangle$ be a Sylow $p_{i_j}$-subgroup of $H$; note that $P_j\not\leq H'$ since $P_j$ acts nontrivially, see Lemma~\ref{lemDivCond}. We need to consider actions $\sigma\colon H\to \Aut(Q)$ that map each $x_j$ to an automorphism of $Q$ of order $p_{i_j}^{e_j}$. To construct all  corresponding extensions $H\ltimes_\sigma Q$ up to isomorphism, Remark~\ref{remInduced} shows that one has to determine canonical $\Aut(H)$-orbit representatives of suitable $P_j$-actions on $Q$, for each $j$; here ``canonical'' refers to the description in Remark~\ref{remPowers}. Note that in order to have a well-defined $\Aut(H)$-action, here we consider $P_j\cong P_j H'/H'$ as a Sylow $p_{i_j}$-subgroup of $H/H'$, see Remark \ref{remInduced}. The description below shows that the $\Aut(H)$-action on group actions $H\to\Aut(Q)$ can be considered separately for each prime in $\mathcal{A}$. 

Thus, in the following let $p\in\mathcal{A}$, say $(p,p_v,e)\in\mathcal{P}$, and let $P=\langle x\rangle$ be a Sylow $p$-subgroup of $H$. Note that  since $p$ acts in $\mathcal{P}$, it is not acted upon in $\mathcal{P}$. We need another case distinction.

If $p_v$ is the smallest prime that $p$ acts on in $\mathcal{P}$, then $p$ does not act in $H$, hence $P\leq H$ is central, and so $H=U\times P$ for some complement $U$. In particular, $\Aut(H)=\Aut(U)\times\Aut(P)$, so $\Aut(H)$ acts as $\Aut(P)$ on $P$. Since $\Aut(P)$ acts transitively on the generators of $P$, it follows that there is a unique $\Aut(H)$-orbit of nontrivial $P$-actions on $Q$ with exponent $e$.

If $p_v$ is not the smallest prime that $p$ acts on in $\mathcal{P}$, then $p$ acts on some prime $p_w$ in $H$, say with exponent $c_w$. Let $ S = \langle y \rangle$ be the Sylow $p_w$-subgroup of $H$; recall that $S\unlhd H'$ and so $S$ is characteristic in $H$.  For every $\beta\in\Aut(H)$, we have that $x^\beta$ acts like $x^t$ for some $t$ coprime to $p$; specifically, $t$ is determined by $(xH')^\beta=x^tH'$. Furthermore, $\beta$ induces an automorphism of $S$ that commutes with $x$ because $\Aut(S)$ is abelian. Thus, we have \[ (y^\beta)^{x} = (y^x)^\beta = (y^\beta)^{x^\beta}= (y^\beta)^{x^t};\]
since $y\mapsto y^x$ has order $p^{c_w}$, it follows that $t\equiv 1\bmod p^{c_w}$, and so $\Aut(H)$ acts trivially on $\Hom(P,\Aut(S))$. This holds for all primes that $p$ acts on in $H$, so  $t\equiv 1\bmod p^c$, where $c$ is the largest exponent among those  $p$-actions. We make a case distinction on the exponent in $(p,p_v,e)$.

By what is shown in the previous paragraph, if $e \leq c$, then every action $\sigma\colon H\to\Aut(Q)$ with $p$ acting with exponent $e$ satisfies
\[ \sigma^{\beta^{-1}}(x)  = \sigma(x^\beta) = \sigma(x^t) = \sigma(x)^t = \sigma(x),\]
and so the $\Aut(H)$-action on $\Hom(P, \Aut(Q))$ is trivial. Thus, there are $p^{e-1}(p-1)$ orbits of $\Aut(H)$ on actions $P\to\Aut(Q)$ with exponent $e$.

It remains to consider  $e>c$. If $\sigma\colon P\to\Aut(Q)$ with $\sigma(x)$ of order $p^e$ and $\beta\in\Aut(H)$ with $\beta(xH')=x^tH'$ as above, then $\sigma^{\beta^{-1}}(x)  = \sigma(x^\beta) = \sigma(x^t) = \sigma(x)^{t \bmod p^e}$. Thus, every $P$-action on $Q$ in the same $\Aut(H)$-orbit as $\sigma$ has the form $\sigma^s$ for some $s\in\{1,\ldots,p^{e}-1\}$ with $s\equiv 1\bmod p^c$. Let $S=P_1$, $S_2,\ldots,S_{v-1}$ be Sylow subgroups of $H$ corresponding to $p_1,\ldots,p_{v-1}$. We claim that for each such value $s$ there is an automorphism $\gamma$ of $H$ that takes $x$ to $x^s$ and fixes $S_2,\ldots,S_{v-1}$. Note that $H$ is generated by $S_1,\ldots,S_{v-1}$, so to define $\gamma$ it suffices to describe how it acts on the cyclic generators $x,z_2,\ldots,z_{v-1}$ of these  subgroups.  By von Dyck's Theorem 
\cite[Theorem~2.53]{handbook}, the map $\gamma\colon (x,z_2,\ldots,z_{v-1})\mapsto (x^s,z_2,\ldots,z_{v-1})$ extends to an automorphism of $H$ if and only if the conjugacy relations for these generators are preserved by $\gamma$. This is trivially true for all conjugacy relations between generators $z_i$ and $z_j$; likewise, if $z_i$ commutes with $x$, then $\gamma(z_f)\gamma(x)= \gamma(x)\gamma(z_f)$. Lastly, suppose $x$ acts non-trivially on $z_i$, say $z_i x=xz_i^j$. Because $p$ acts at most with exponent $c$ in $H$ and because $s\equiv 1\bmod p^c$, we have $z_f^{x^s} = z_f^{x}$, and so $\gamma(z_j)\gamma(x) = z_jx^s=x^sz_j^{x^s}=x^sz_j^{x}=x^sz_j^j= \gamma(x)\gamma(z_j)^j$, as required. This proves that  $\gamma$ extends to an automorphism of $H$ that maps $x$ to $x^s$ and fixes the other Sylow subgroups $S_2,\ldots,S_{v-1}$. As there are $p^{e-c}$ such numbers $s$, the $\Aut(H)$-orbit of the above $\sigma\colon P\to\Aut(Q)$ has size $p^{e-c}$. Note that there are $p^{e-1}(p-1)$ elements in $\Aut(Q)$ of order $p^e$, thus in the case discussed in this paragraph ($e>c$) the number of distinct $\Aut(H)$-orbits of actions of $P$ on $Q$ with exponent $e$ is $p^{e-1}(p-1)/p^{e-c} =  p^{c-1}(p-1)$.

In  conclusion, we can determine the canonical actions $\sigma\colon H\to\Aut(Q)$ by considering each prime $p\in\mathcal{A}$ separately; this provides a partial converse to the last statement in Remark~\ref{remInduced}. As a corollary, our description also leads to an immediate counting procedure for  the number of isomorphism types of groups of order $n$ that have cluster $\mathcal{P}$.

\subsection{C-groups by ID}\label{secConstID}
We sort the C-groups of order $n>1$ first by their acting divisors (in increasing order), then by the size of their acting groups (in increasing order). Thus, to construct the C-group $G$ with ID $(n,i)$, we proceed as follows. First, we determine  $j\geq 0$ such that \[C_{d_1}(n)+\ldots+C_{d_{j-1}}(n)<i\leq C_{d_1}(n)+\ldots+C_{d_{j}}(n),\] so $G$ must have acting divisor $d=d_j$. Now it remains to construct the $t$-th group with acting divisor $d$, where  $t=i-C_{d_1}(n)-\ldots-C_{d_{j-1}}(n)$. If  we have prime-power factorisations $d=p_1^{a_1}\ldots p_u^{a_u}$ and $n/d=q_1^{b_1}\ldots q_v^{b_v}$, then every acting group order $m$ is divisible by $p_1\ldots p_u$ and divides $\gcd(d,(q_1-1)\ldots (q_v-1))$. Let $m_1<\ldots<m_r$ be all those possible acting group orders and  determine  $j\geq 1$ such that \[C_{d,m_1}(n)+\ldots +C_{d,m_{j-1}}(n)<t\leq C_{d,m_1}(n)+\ldots +C_{d,m_j}(n),\] so $G$ must have an acting group of order $m=m_j$. Now it remains to construct the $\ell$-th group with acting divisor $d$ and acting group of order $m$, where $\ell=t-C_{d,m_1}(n)-\ldots -C_{d,m_{j-1}}(n)$. Let $\mathcal{P}_1,\ldots,\mathcal{P}_w$ be all the possible clusters for a C-group of size $n$, with acting divisor $d$ and acting group of order $m$, sorted as described in Section \ref{secSortCluster}, and  let $C_{\mathcal{P}_c}(n)$ denote the number of C-groups of order $n$ with cluster $\mathcal{P}_c$; as described in Section \ref{secGenP}, this number can be calculated from $\mathcal{P}_c$.  We determine $j\geq 1$ such that 
\[C_{\mathcal{P}_1}(n) +\ldots+C_{\mathcal{P}_{j-1}}(n) < \ell \leq C_{\mathcal{P}_1}(n) + \ldots + C_{\mathcal{P}_{j}}(n),\] so $G$ must have cluster $\mathcal{P}=\mathcal{P}_j$. Now it remains to construct the $s$-th group of order $n$ with this cluster, where $s = \ell - C_{\mathcal{P}_1}(n) - \ldots - C_{\mathcal{P}_{j-1}}(n-1)$. For each triple in $\mathcal{P}$, the description in Section \ref{secGenP} tells us how many non-isomorphic choices of actions there are for $p$ acting on $q$ with exponent $e$.  Thus, with $\mathcal{P}$ sorted first by the acting primes then by the primes which are acted on, we can determine the corresponding action for each triple; the description in Remark \ref{remPowers} makes this canonical, hence the group with ID $(n,i)$ is determined.

\subsection{Identification of C-groups}\label{secId}
\noindent For a given C-group $G$ of order $n$, we determine its ID $(n,i)$ such that $G$ is isomorphic to the group with ID $(n,i)$ as constructed in Section \ref{secConstID}. For this, we compute the Sylow subgroups of $G$, and from those determine the  cluster $\clust(G)$. This also provides the decomposition $G=(A\ltimes D)\times C$ as in Lemma \ref{lemDivCond}, the acting divisor $d=|A|$, the acting group $K\leq \Aut(D)$ and its order $m=|K|$. If $d_1<\ldots<d_j=d$ are the possible acting divisors at most  $d$, and $m_1<\ldots< m_r=m$ are the possible acting group orders (of a C-group of order $n$ with acting divisor $d$) at most $m$, then the ID $(n,i)$ of $G$ satisfies $C<i\leq C+C_{d,m}(n)$ where 
\[C=C_{d_1}(n)+\ldots+C_{d_j-1}(n)+C_{d,m_1}(n)+\ldots+C_{d,m_{r-1}}(n).\] 
It remains to determine the position of $G$ within the clusters for $C_{d,m}(n)$; those clusters are sorted (Section \ref{secSortCluster}) and for each cluster we know how many isomorphism types of group exist (Section~\ref{secGenP}). The ID of $G$ is now determined  by reversing the normalisation process  in Section~\ref{secGenP}.

\section{Practical performance}\label{secGAP} 
\noindent A GAP implementation of our algorithms  is  available online (see \cite{code}) and provides functions {\tt AllCGroups}, {\tt CGroupById}, and {\tt IdCGroup};  we comment on the performance below. All computations have been carried out with GAP 4.10.0 on a computer with 16GB RAM and Intel(R) Core(TM) i5-7500 CPU@3.40GHz. We checked the correctness  against GAP's SmallGroups Database and a brute-force implementation based on Remark \ref{remBF}. All groups we considered were constructed or given in GAP as PC-groups (polycyclic groups).

\subsection{Squarefree groups} Since every group of squarefree order is a C-group, we first compare our implementation against the inbuilt GAP functions {\tt AllSmallGroups} and {\tt IdSmallGroup}. Up to order 250000 there are 566801 isomorphism types of groups of squarefree order. Our implementation of {\tt AllCGroups} constructed these groups in 201 seconds; {\tt AllSmallGroups} required  2844 seconds. Table \ref{tabSQ} lists some runtimes for some selected squarefree orders that involve larger prime factors (some greater than $10^6$). Overall, our experiments indicate that {\tt AllCGroups} is often  a factor ten faster than GAP's {\tt AllSmallGroups}.

{\tt AllCGroups} and {\tt AllSmallGroups} constructed the 208014 groups of squarefree order at most 100000 in 75 seconds and in 536 seconds, respectively. Taking GAP's list of those groups, our  {\tt IdCGroup} needed 161 seconds to compute all their IDs, whereas GAP's {\tt IdSmallGroup} required 267 seconds to compute the IDs of our groups. We note that due to different orderings of groups, our ID is not the same as the ID in the SmallGroups Library.

Table \ref{tabSQ} lists some further runtimes. In general, the performance of {\tt IdCGroup} seems  comparable with GAP's {\tt IdSmallGroup}; for both functions, when group orders with large prime divisors are involved, a main bottleneck seems to be GAP's functionality to compute Sylow subgroups. We have not included runtimes in Table~\ref{tabSQ} when the impact of the latter dominated the ID computation, but rather wrote `syl'.

\subsection{General orders} There are 576093 isomorphism types of C-groups of order at most 100000. It took 247 seconds to construct all these groups with {\tt AllCGroups}. Applying {\tt IdCGroup} to (isomorphic copies of) these groups took 846 seconds, that is, on average  0.0015 seconds per group. In Table \ref{tab2} we list runtimes for constructing all C-groups of  larger orders, and for determining their IDs; the average runtime for {\tt IdCGroup} was at most 0.013 seconds per group.

\newcommand{\mmc}[1]{\multicolumn{1}{|r|}{#1}}
\newcommand{\mmcc}[1]{\multicolumn{1}{r|}{#1}}

\renewcommand\arraystretch{1.2}
\begin{table}[h]
  {
  \hspace*{-2.2cm}\begin{tabular}{rrrrrr}\label{tabSQ}  
    squarefree group order & \# groups & \rot{{\tt AllCGroups}} & \rot{{\tt AllSmallGroups}} & \rot{{\tt IdCGroup}} & \rot{{\tt IdSmallGroup}}\\\hline
    4140806021907601450474046095 &  \mmc{126} & 0.1 & \mmcc{1.4} & 51.8 & 61.3\\
    1054578325689038795758113    & \mmc{299} & 0.1 & \mmcc{1.6} & {121} & 156\\
   18246294181628283634185      & \mmc{1678} & 0.6 & \mmcc{2.7} & {80.3}  & 95.2  \\ 
   288580601323668153539638920527445 & \mmc{110} & 0.06 & \mmcc{6.5} &syl & syl\\
    4100698523844820373769891971054 & \mmc{1024} & 0.55 & \mmcc{15.6} & syl & syl \\
    24898143467617960290            & \mmc{384} & 0.2 & \mmcc{175.7}  & syl  & syl  \\
     2533036924228662499419966    & \mmc{3840} & 1.9 & \mmcc{49.8} &  syl & syl   
  \end{tabular}\\[1ex] \caption{Some squarefree orders with large prime factors: we list the time (in seconds) required to construct and recognise \emph{all} groups of the given order.}
  
  \begin{tabular}{rrrr}\label{tab2}
    factorised group order & \# groups & \rot{{\tt AllCGroups}} &  \rot{{\tt IdCGroup}} \\\hline
$5^5. 7^5. 11^5. 13^5. 197^7. 251^4. 677^8. 727^4$ & \mmc{225} & 0.9 & 9.9 \\    
$2^2.  31^2.  113^3.  227^4.  293^4.  373$    & \mmc{276} & 0.3 &  1.9   \\
$2^5.  3^5.  101^3.  103.  313^2.  367^5$      & \mmc{840} & 1.3 & 8.5   \\
$2^3.  173^2.  233^4.  241^2.  307^2.  337^2$  & \mmc{1168} & 1.1 &  10.0   \\
$3^3.   5^3.   7^2.   11^3.   23^2.   43^2.   101^2.   127^2$ & \mmc{1305} & 0.9 & 6.3      \\
$2^4.  5^4.  73.  101^2.  113^3.  349^3$   & \mmc{2720} & 2.1 & 17.5      \\
$2^2.  3^5.  5.  61^2.  73^5.  349^4$      & \mmc{4128} & 5.1 &  38.2     \\
$3^3.  5^2.  7^3.  29^3. 59^2. 233^3. 43^3. 173^3. 431^2$ & \mmc{6006} & 4.6 & 73.9  
  \end{tabular}\\[1ex] \caption{Some non-squarefree orders: we list the time (in seconds) required to construct all C-groups of the given order and to recognise \emph{all} those groups.}}
\end{table}

\bibliographystyle{line}

\end{document}